\documentclass[12pt,a4paper]{article}
\usepackage[utf8]{inputenc}
\usepackage[english]{babel}
\usepackage{indentfirst}
\usepackage{misccorr}
\usepackage{graphicx}
\usepackage{amsmath}
\usepackage{amsthm}
\usepackage{bm}
\usepackage[matrix,arrow,curve]{xy}
\def\d{\partial}

\theoremstyle{definition}
\newtheorem{mdef}{Definition}[section]
\newtheorem{mnot}{Notation}[section]

\newtheorem{mrem}{Remark}[section]
\theoremstyle{plain}
\newtheorem{mth}{Theorem}[section]
\newtheorem{mst}{Statement}[section]
\newtheorem{mcor}{Corollary}[section]
\newtheorem{mlm}{Lemma}[section]
\newtheorem{mprop}{Proposition}[section]
\author{Mingazov A. A.}

\title{Some remarks on relative framed motives}
\begin{document}
\maketitle

\section{Introduction}
The purpose of this paper is to prove two statements which are key in the construction of mnotivic fibrant replacement for suspension spectrum conserned with motivic spece $\dfrac{X}{X-Z}$,  where $X$ is smooth variety and $Z$ is its smooth subvariety. In fact this generalize the main result of the paper \cite{GNP}.The main results a formulated in theorems \ref{iso_geom_factor} and \ref{cancelcation}.

\section{Basic notions}
\begin{mdef}
Let $Y$ be a $k$-smooth scheme, $S\subset Y$ be a closed subset and $U\in Sm/k$. A framed correspondence of level $m$ from $U$ to  $Y/(Y-S)$ is the set of data
$$(Z,W,\phi_1,\ldots,\phi_m; g\colon W\rightarrow Y),$$
where $Z\subset U\times\mathbb{A}^m$ is closed and finite over $U$, $W$ is etale neighborhood of $Z$ in $U\times \mathbb{A}^m$, $\phi_1,\ldots, \phi_m$ are regular functions on $W$, $g$ is a regular map such as $Z=g^{-1}(S)\cap V(\phi_1,\ldots, \phi_m)$.

Two correspondence are equivalent if they coiunside on the common etale neighborhood.
\end{mdef}

\begin{mnot}

1) $Fr_m(U,Y/(Y-S))$ is the set of framed correspondencefrom $U$ to $Y/(Y-S)$ up to equivalence.

2) $F_m(U,Y/(Y-S))$ is is the set of framed correspondencefrom $U$ to $Y/(Y-S)$ up to equivalence with connected support $Z$.

3) $Fr_*(U,Y/(Y-S))=\mathop\bigsqcup\limits_{m\geq 0}Fr_m(U,Y/(Y-S))$.

4) $Fr(U,Y/(Y-S))$ is the set stabilized up to the map \break $\Sigma\colon \Phi\mapsto \Phi\boxtimes \sigma$, where $\sigma=(\{0\},\mathbb{A}^1, id\colon \mathbb{A}^1\rightarrow\mathbb{A}^1, const\colon \mathbb{A}^1\rightarrow pt)\in Fr_1(pt,pt)$.

5) $\mathbb{Z}Fr_m(U,Y/(Y-S))$ is the free abelian group generated by $Fr_m$.

6) $\mathbb{Z}F_*(U,Y/(Y-S))$ is the factor of $\mathbb{Z}Fr_m(U,Y/(Y-S))$ up to the relation:
$$(Z\sqcup Z',W,(\phi_1,\ldots,\phi_m); g\colon W\rightarrow Y)=$$
$$=(Z,W\setminus Z',(\phi_1,\ldots,\phi_m)|_{W\setminus Z'}; g|_{W\setminus Z'})+(Z,W\setminus Z,(\phi_1,\ldots,\phi_m)|_{W\setminus Z}; g|_{W\setminus Z})$$

Note that $F_m(U,Y/(Y-S))$ is a basis of the group $\mathbb{Z}F_m(U,Y/(Y-S))$.
\end{mnot}

\begin{mdef}
Presheaf $\mathcal{F}$ with $\mathbb{Z}F_*$-transfers is a functor $\mathcal{F}\colon \mathbb{Z}F_*^{op}\r Ab$. The category of Nisnevich shaeves with  $\mathbb{Z}F_*$-transfers wil be denoted as $NSZF_*$.
\end{mdef}

\begin{mlm}
$\mathbb{Z}F_*(-,X)$ is a Nisnevich sheaf.
\end{mlm}

\section{Some statements on cohomologies of sheaves with $\mathbb{Z}F_*$-transfers}
The statements of this section are modifications of similar theorems of \cite{SV}.

\begin{mlm} Let $f\colon Y\rightarrow X$ be a Nisnevich covering of variety $X$. Then the sequence of the sheaves
$$0\leftarrow\mathbb{Z}F_n(X)\xleftarrow{f_*}\mathbb{Z}F_n(Y)\xleftarrow{(p_2)_*-(p_1)_*}\mathbb{Z}F_n(Y\times_X Y)\leftarrow\ldots$$
is exact.
\end{mlm}
\begin{proof}
Let $U$ be a local Henselian ring. We need to prove the exactness of the complex
$$A_*(U)=\left(0\leftarrow\mathbb{Z}F_n(U,X)\xleftarrow{f_*}\mathbb{Z}F_n(U,Y)\xleftarrow{(p_2)_*-(p_1)_*}\mathbb{Z}F_n(U,Y\times_X Y)\leftarrow\ldots\right).$$

Fix framed correspondence $\alpha$
$$\xymatrix{&W\ar[dl]\ar[r]&X\times \mathbb{A}^n\\
U\times \mathbb{A}^n &Z\ar[r]\ar@{_(->}[l]\ar[d]\ar@{^(->}[u]&X\times 0,\ar@{^(->}[u],\\
&U&
}
$$
where scheme $Z$ is connected.

Let $A^\alpha_k(U)$ be a subgroup of $A_k(U)$ generated by correspondences $\beta$, which after composition with map$Y\times_X Y\times\ldots \times_X Y\rightarrow Y$ are divisible by $\alpha$. Note that $A^\alpha_*(U)$ is subcomplex in  $A_*(U)$ and $A_*=\bigoplus\limits_\alpha A^\alpha_*(U)$. In particular, the support of all elements of $A^\alpha_*(U)$ is $Z$. It is enough to prove that the complex $A^\alpha_*(U)$ is contractible. 

As $Z$ is connected and finite over local Henselian ring $U$, it is local and Henselian. Consider two Cartesian squares
$$\xymatrix{Z\times_X Y\ar[r]\ar[d]&W\times_X Y\ar[r]\ar[d]&Y\ar[d]^f\\
Z\ar[r]&W\ar[r]&X.}$$
The map $Z\times_XY\rightarrow Z$ is etale as the pullback of the etale map $f\colon Y\rightarrow X$, so it has a section $s\colon Z \rightarrow Z\times_X Y$ as $Z$ is Henselian. Let fix it. As $W$ is Henselian too, we will suppose thar there exist some underline diagramm $Z\hookrightarrow W\xrightarrow{g} Y$.

Now we construct a contracting homotopy $h_k^\alpha\colon A_k^\alpha(U)\rightarrow A_{k+1}^\alpha(U)$. Let us take $\beta\in A_k^\alpha(U)$
$$\xymatrix{&W\ar[dl]\ar[rr]^{q~~~~~~~~~} &&Y\times_X\ldots\times_X Y\times \mathbb{A}^n\\
U\times\mathbb{A}^n\ar[dr]&Z\ar[d]\ar@{^(->}[u]\ar@{_(->}[l]\ar[rr]&&Y\times_X\ldots\times_X Y\times 0.\ar@{^(->}[u]\\
&U&&}$$
Then $h(\beta)$ can be defined as
$$\xymatrix{&W\ar[dl]\ar[rr]^{q\times_X g~~~~~~~~~~~~~~} &&Y\times_X\ldots\times_X Y\times_X Y\times \mathbb{A}^n\\
U\times\mathbb{A}^n\ar[dr]&Z\ar[d]\ar@{^(->}[u]\ar@{_(->}[l]\ar[rr]&&Y\times_X\ldots\times_X Y\times_X Y\times 0.\ar@{^(->}[u]\\
&U&&}$$
The properties of contracting homotopy can be verified straightly.
\end{proof}

\begin{mst}
Let $X$ be a smooth variety. Then presheaves $H^i(C^*\mathbb{Z}F(X))$ are quasistable.
\end{mst}

\begin{mcor}
Let $I$ be an injective Nisnevich sheaf with $\mathbb{Z}F_*$-transfers. Then for any $X\in Sm/k$  $$H^i_{Nis}(X,I)=0.$$
\end{mcor}

\begin{mcor}
For any Nisnevich sheaf $\mathcal{F}$ with  $\mathbb{Z}F_*$-transfers
$$Ext^i_{NSZF_*}(\mathbb{Z}F_*(X),\mathcal{F})=H^i_{Nis}(X,\mathcal{F}).$$
\end{mcor}

\begin{mprop}
Let $A^*\in D^-(NSZF_*)$ be a bounded compex of Nisnevich sheaves with $\mathbb{Z}F_*$-transfers. Then for any $X\in Sm/k$
$$H^i_{Nis}(X,A^*)=Hom_{D^-(NSZF_*)}(\mathbb{Z}F_*(X),A^*[i]).$$
\end{mprop}

\begin{mdef}
Presheaf is called contractible if there exist the preashef morfism  $\phi\colon \mathcal{F}(-)\rightarrow \mathcal{F}(\Delta^1\times -)$ such as $\d_0\phi=0$, $\d_1\phi=1_{\mathcal{F}}$.
\end{mdef}

\begin{mprop}
Let $\mathcal{G}$, $\mathcal{F}$ be quasistable Nisnevich sheaves with $\mathbb{Z}F_*$-transfers. Suppose that $\mathcal{G}$ is contractible, $\mathcal{F}$ is strongly homotopy invariant. Then $Ext^i_{NSZF_*}(\mathcal{G},\mathcal{F})=0$ for all $i\geq 0$.
\end{mprop}

\begin{mcor}
Let $\mathcal{G}$, $\mathcal{F}$ be quasistable Nisnevich sheaves with $\mathbb{Z}F_*$-transfers. Suppose that  $\mathcal{F}$ is strongly homotopy invariant and $\mathcal{G}$ has a finite resolution $$0\rightarrow \mathcal{G}\rightarrow  \mathcal{G}^0\rightarrow\ldots \rightarrow  \mathcal{G}^n\rightarrow 0,$$
where $\mathcal{G}^i$ is contractible and quazistable  Nisnevich sheaf with $\mathbb{Z}F_*$-transfers for all $i$. Then $Ext^i_{NSZF_*}(\mathcal{G},\mathcal{F})=0$.
\end{mcor}

\begin{mcor}
Let $A^*$ be a bounded complex of quazistable Nisnevich sheaves with $\mathbb{Z}F_*$-transfers, $A^i$ is contractible for all $i$, cohomology preasheaves $\mathcal{H}^i=H^i(A^*)$ are homotopy invariatn and quasistable. Then the complex $A^*$ is acyclic.
\end{mcor}

\begin{mcor}\label{C_sh}
Let $A^*$ ba a complex of Nisnevich sheaves with $\mathbb{Z}F_*$-transfers which satisfies the following conditions:

$\mathrm{1)}$ $A^*_{Nis}=0;$

$\mathrm{2)}$ presheaves $H^i(C^*(A^*))$ are quasistable.

Then  $C^*(A^*)$ is locally acyclic.
\end{mcor}

\section{Mayer-Vietoris sequence and etale excision for framed motives}

\begin{mst} \label{C} Let $0\rightarrow\mathcal{F}_1\rightarrow \mathcal{F}_2\rightarrow \mathcal{F}_3 \rightarrow 0$ be an exact sequence of Nisnevich sheaves with $\mathbb{Z}F_*$-tranfers and cohomology preasheaves $H^i(C^*(\mathcal{F}_r))$ are quasistable for all $i$, $r$. Then 
$$C^*(\mathcal{F}_1)\rightarrow C^*(\mathcal{F}_2)\rightarrow C^*(\mathcal{F}_3)\rightarrow C^*(\mathcal{F}_1)[1]$$
is distinguished triangle in $D^-(NSZF_*)$.
\end{mst}

\begin{mst}
Let $A, B, C, D$ be sets. Consider the diagram
$$\xymatrix{A\ar[r]^{\alpha}\ar[d]^{\beta}& B\ar[d]^{\delta}\\
C\ar[r]^{\gamma}& D}$$
and the sequense of abelian groups
$$0\rightarrow\mathbb{Z}[A]\xrightarrow{\begin{pmatrix}\alpha\\-\beta\end{pmatrix}}
\begin{array}{ccc}\mathbb{Z}[B]\\ \oplus \\\mathbb{Z}[C]\end{array}\xrightarrow{\begin{pmatrix}\delta &\gamma\end{pmatrix}}\mathbb{Z}[D]\rightarrow 0.$$
Then

$\mathrm{1)}$ the sequence is the complex if and only if the diagram comutes;

$\mathrm{2)}$  the sequence is exact on the left  if and only if the diagram is Cartesian square;

$\mathrm{3)}$    the sequence is exact on the right if and only if the diagram is pushout square.
\end{mst}

\begin{mlm}
Let $U$ be a spectrum of Henselian ring, $X\subset Y$ be an open imbedding. Then the map
$F_n(U,X)\rightarrow F_n(U,Y)$
is injective.
\end{mlm}

\begin{mlm}\label{MV_s}
Let $X=X_1\cup X_2$ be a union of its open subsets, $Z\subset X$ be a closed subvariety. Denote $X_{12}=X_1\cap X_2$, $Z_1=Z\cap X_1$, $Z_2=Z\cap X_2$, $Z_{12}=Z\cap X_1\cap X_2$. Let $v_i\colon X_{12}\rightarrow X_i$ be inclusion $i=\overline{1,2}$, $u_i\colon X_i\rightarrow X$, $i=\overline{1,2}$. Then the sequence of Nisnevich sheaves
$$0\rightarrow\mathbb{Z}F_n \left(\dfrac{X_{12}}{X_{12}-Z_{12}}\right)\xrightarrow{\begin{pmatrix}v_{1*}\\-v_{2*}\end{pmatrix}}
\begin{array}{ccc}\mathbb{Z}F_n\left(\dfrac{X_{1}}{X_{1}-Z_{1}}\right)\\ \oplus \\ \mathbb{Z}F_n\left(\dfrac{X_{2}}{X_{2}-Z_{2}}\right)\end{array}\xrightarrow{\begin{pmatrix}u_{1*}&u_{2*}\end{pmatrix}}\mathbb{Z}F_n\left(\dfrac{X}{X-Z}\right)\rightarrow 0$$
is exact.
\end{mlm}

\begin{mdef}
Let $X$ be a smooth variety and $Z\subset X$ is its clozed subset. Let us denote

$$M_Z^{fact}X=\dfrac{C^*\mathbb{Z}F(X)}{C^*\mathbb{Z}F(X-Z)},$$
$$M_Z^{geom}X=C^*\mathbb{Z}F\left(\dfrac{X}{X-Z}\right).$$

Note that there exist the natural map
$$M_Z^{fact}(X)\rightarrow M_Z^{geom}(X).$$
\end{mdef}

\begin{mst}\label{MV} Let the variety $X=X_1\cup X_2$ be an union of its open subsets,  $Z\subset X$ be its clozed subvariety. Denote $X_{12}=X_1\cap X_2$, $Z_1=Z\cap X_1$, $Z_2=Z\cap X_2$, $Z_{12}=Z\cap X_1\cap X_2$.  Then the triangles 
$$\mathrm{1)~} M_{Z_{12}}^{fact}(X_{12})\rightarrow M_{Z_{1}}^{fact}(X_1) \oplus M_{Z_{2}}^{fact}(X_2)\rightarrow M_{Z}^{fact}X\rightarrow M_{Z_{12}}^{fact}(X_{12})[1], $$

$$\mathrm{2)~} M_{Z_{12}}^{geom}(X_{12})\rightarrow M_{Z_{1}}^{geom}(X_1) \oplus M_{Z_{2}}^{geom}(X_2)\rightarrow M_{Z}^{fact}X\rightarrow M_{Z_{12}}^{geom}(X_{12})[1]$$
are distinguished.
\end{mst}

\begin{mcor}
Let $\pi\colon L\rightarrow X$ ba a vector bundle over $X$, $X_1,X_2\subset X$ be open subsets, $X_{12}=X_1\cap X_2$, $L_i=L|_{X_i}$. Then the triangles 
$$M^{fact}_{X_{12}}(L_{12})\rightarrow M^{fact}_{X_{1}}(L_{1})\oplus M^{fact}_{X_{2}}(L_{2})\rightarrow M^{fact}_{X}(L),$$
$$M^{geom}_{X_{12}}(L_{12})\rightarrow M^{geom}_{X_{1}}(L_{1})\oplus M^{geom}_{X_{2}}(L_{2})\rightarrow M^{geom}_{X}(L)$$
are distinguished in $D^-(NSZF_*)$. 
\end{mcor}

\begin{mth} \label{bundle_iso}
Let $p\colon L\rightarrow X$ be a vector bundle. Then the map
$$M^{fact}_X(L)\rightarrow M^{geom}_X(L)$$
is the local isomorphism.
\end{mth}
\begin{proof}
We'll use the induction on the number of open sets covering $X$ on which $L$ is trivial. If $L$ is trivial on $X$, the statement of the theorem is the main result of \cite{GNP}. Let $X_1,X_2\subset X$ be open subsets and $X_1\cup X_2=X$. If $L_1=L|_{X_1}$ and $L_2=L|_{X_2}$ are trivial, then $L_{12}=L|_{X_{12}}$ is trivial too. Consider the diagramm
$$\xymatrix{C^*\mathbb{Z}F\left(\dfrac{L_{12}}{L_{12}-X_{12}}\right)\ar[r]&C^*\mathbb{Z}F\left(\dfrac{L_1}{L_1-X_1}\right)\oplus C^*\mathbb{Z}F\left(\dfrac{L_2}{L_2-X_2}\right)\ar[r]&C^*\mathbb{Z}F\left(\dfrac{L}{L-X}\right)\\
\dfrac{C^*\mathbb{Z}F(L_{12})}{C^*\mathbb{Z}F(L_{12}-X_{12})}\ar[r]\ar[u]&\dfrac{C^*\mathbb{Z}F(L_1)}{C^*\mathbb{Z}F(L_1-X_1)}\oplus\dfrac{C^*\mathbb{Z}F(L_2)}{C^*\mathbb{Z}F(L_2-X_2)}\ar[r]\ar[u]&\dfrac{C^*\mathbb{Z}F(L)}{C^*\mathbb{Z}F(L-X)}\ar[u]\\
}$$
String are distinguished triangles and two vertical maps are isomorphisms in $D^-(NSZF_*)$.
\end{proof}

\begin{mst} Let $f\colon X'\rightarrow X$ be etale morphism, $Z\subset X$ be a closed subset and $f\colon f^{-1}(Z)\rightarrow Z$ be an isomorphism. Then

$\mathrm{1)}$ the map $f^{fact}_*\colon M_Z^{fact}*(X')\rightarrow M_Z^{fact}(X)$ is isomorphism in $D^-(NSZF_*)$,

$\mathrm{2)}$ the map $f^{geom}_*\colon M_Z^{geom}(X')\rightarrow M_Z^{geom}(X)$ is isomorphism in $D^-(NSZF_*)$.
\end{mst}

\section{Deformation to normal bundle}

\begin{mst}
Let $X$ be a smooth variey, $Z$ is its closed subset. Denote with $p\colon X\times \mathbb{A}^1\rightarrow X$ projection.Then the map $$p_*\colon M_{Z\times \mathbb{A}^1}^{geom}(X\times \mathbb{A}^1)\rightarrow M_Z^{geom}(X)$$ is isomorphism in $D^-(NSZF_*)$.
\end{mst}

\begin{mcor}
Let $p\colon E\rightarrow X$ be a vector bundle over $X$, $Z\subset X$, $S=p^{-1}(Z)$. Then the map $$p_*\colon M_{S}^{geom}(E)\rightarrow M_Z^{geom}(X)$$ is isomorphism in $D^-(NSZF_*)$.
\end{mcor}
\begin{proof}This follows from the previous statement by applying \ref{MV}.
\end{proof}

\begin{mst}
Let $X$ be a smooth variety and $Z$ is its closed subset. Denote with $p\colon X\times \mathbb{A}^1\rightarrow X$  the projection. Then the map $$p_*\colon M_{Z\times \mathbb{A}^1}^{fact}(X\times \mathbb{A}^1)\rightarrow M_Z^{fact}(X)$$ is isomorphism in $D^-(NSZF_*)$.
\end{mst}
\begin{proof}
This follows from the homotopical invariance of $C^*\mathbb{Z}F(X)$.
\end{proof}

\begin{mcor}
Let $p\colon E\rightarrow X$ be a vector bundle over X, $Z\subset X$, $S=p^{-1}(Z)$. Then the map $$p_*\colon M_{S}^{fact}(E)\rightarrow M_Z^{fact}(X)$$ is an isomorphism in $D^-(NSZF_*)$.
\end{mcor}
\begin{proof}
This follows from the previous statement by applying \ref{MV}.
\end{proof}

\begin{mst}
Let $X$ be a smooth variety, $Z\subset X$ be its closed subset and $N=N_{X/Z}$ be the normal bundle. Denote with $X_t$ deformation to normal bundle. Let $U$ and $V$ be open subsets of$X$. Then

a) $U_t\cap V_t=(U\cap V)_t$;

b) $U_t\cup V_t=(U\cup V)_t$;

c) Let
$$\xymatrix{\widetilde{X}-\widetilde{Z}\ar[d]^{e|_{\widetilde{X}-\widetilde{Z}}} \ar@{^(->}[r]& \widetilde{X}\ar[d]^{e}&\widetilde{Z}\ar@{_(->}[l]\ar@{=}[d]\\
X-Z\ar@{^(->}[r]& X&Z\ar@{_(->}[l]}$$
be an elemental Nisnevich square. Then
$$\xymatrix{\widetilde{X}_t-\widetilde{Z}\times\mathbb{A}^1\ar[d] \ar@{^(->}[r]& \widetilde{X}_t\ar[d]^{e_t}&\widetilde{Z}\times\mathbb{A}^1\ar@{_(->}[l]\ar@{=}[d]\\
X_t-Z\times \mathbb{A}^1\ar@{^(->}[r]& X_t&Z\times\mathbb{A}^1\ar@{_(->}[l]}$$
is also elemental Nisnevich square.
\end{mst}

\begin{proof}
This follows from construction of $X_t$ described in \cite{PS}.
\end{proof}

\begin{mth}\label{defor}
Let $X$ be a smooth variety, $Z\subset X$ be its smooth subvarietyand $N=N_{X/Z}$ be the normal bundle. Denote with $X_t$ deformation to normal bundle. The maps of pairs
$$(N, N-Z)\xrightarrow{i_0}(X_t,X_t-Z\times \mathbb{A}^1)\xleftarrow{i_1}(X,X-Z)$$
induce the isomorphisms

$\mathrm{1)}$ $M_Z^{geom}(N)\xrightarrow{\sim}M_{Z\times \mathbb{A}^1}^{geom}(X_t)\xleftarrow{\sim}M_Z^{geom}(X),$

$\mathrm{2)}$ $M_Z^{fact}(N)\xrightarrow{\sim}M_{Z\times \mathbb{A}^1}^{fact}(X_t)\xleftarrow{\sim}M_Z^{fact}(X).$
\end{mth}

\begin{proof}
The proof is similar to the proof of Theorem 1.2 from \cite{PS}.
\end{proof}

  \begin{mlm}
  Let $p\colon E\rightarrow Z$ be a vector bundle over smooth $Z$, $i\colon Z \hookrightarrow E$ be a zoro section. Then the maps of pairs
$$(N, N-Z)\xrightarrow{i_0}(E_t,E_t-Z\times \mathbb{A}^1)\xleftarrow{i_1}(E,E-Z),$$
where the normal bundle $N$ also coinsides with $E$, induce the isomorphisms

$\mathrm{1)}$ $M_Z^{geom}(N)\xrightarrow{\sim}M_{Z\times \mathbb{A}^1}^{geom}(E_t)\xleftarrow{\sim}M_Z^{geom}(E),$

$\mathrm{2)}$ $M_Z^{fact}(N)\xrightarrow{\sim}M_{Z\times \mathbb{A}^1}^{fact})(E_t)\xleftarrow{\sim}M_Z^{fact}(E).$
  \end{mlm}
 \begin{proof}
 The proof is similar to the proof of lemma 2.5 in \cite{PS}.
 \end{proof}

\begin{mth}\label{iso_geom_factor}
Let $X$ be a smooth variety and  $Z\subset X$ be its smooth subvariety. Then the map
$$M_Z^{fact}(X)\rightarrow M_Z^{geom}(X)$$
is the isomorphism in $D^-(NSZF_*)$.
\end{mth}
\begin{proof}
Consider the deformation to normal bundle $X_t$. We have the diagram
$$\xymatrix{M_Z^{fact}(X)\ar[r]^{i_{1*}~~}\ar[d]&M_{Z\times \mathbb{A}^1}^{fact}(X_t)\ar[d]&M_Z^{fact}(N)\ar[l]_{~~i_{0*}}\ar[d]\\
M_Z^{geom}(X)\ar[r]^{i_{1*}~~}&M_{Z\times \mathbb{A}^1}^{geom}(X_t)&M_Z^{geom}(N)\ar[l]_{~~i_{0*}}.}$$
By the theorems \ref{defor} and  \ref{bundle_iso} we get the statement.
\end{proof}

\section{Cancelcation theorem}

\begin{mdef} Let $U$, $X$ be smooth varieties over field $k$, $S\subset X$ be a closed subset. We'll define the set$Fr_n^{qf}\left(U, \dfrac{X}{X-Z}\right)$ as a set of the following data
$$\xymatrix{
&W=(X\times \mathbb{A}^n)_Z^h\ar[ddl]_{e}\ar[rrrr]^{(g, \phi_1,\ldots,\phi_n)}&&&&X\times\mathbb{A}^n\\
&&&Z'=V(\phi_1,\ldots,\phi_n)\ar@{_(->}[ull]_{i}&&\\
U\times\mathbb{A}^n\ar[dr]_{p_U}&Z\ar[d]^{p}\ar@{_(->}[l]\ar[rrrr]\ar@{^(->}[uu]\ar@{^(->}[urr]&&&&S\times0\ar@{^(->}[uu]\\
&U&&&&\\
}$$
where $e:W=(X\times \mathbb{A}^n)_Z^h\rightarrow X\times \mathbb{A}^n$ is Henselian neighborhood of $Z$, $Z$ is  is a set-theoretical type pullback of $S\times 0$ in $W$, besides the composition $p_U\circ e\circ i\colon Z'\rightarrow U$ is quasifinite.

Denote with $F_n^{qf}\left(U, \dfrac{X}{X-S}\right)$ the subset of $Fr_n^{qf}\left(U, \dfrac{X}{X-S}\right)$ with connected supports $Z$.
\end{mdef}

\begin{mlm}\label{geom_lemma}
Let $V$ be an affine scheme, $Z\subset V$ be a connected closed subset, $can=can_{V,Z}\colon V^h_Z\rightarrow V$ be a Henselisation of $V$ in $Z$, $s\colon Z\rightarrow V^h_Z$ be a section $can$ over $Z$. Let $U$ be a regular local Henselian scheme, $q\colon V \rightarrow U$ be a smooth morphism such that $q|_Z\colon Z\rightarrow V$ is finite. Suppose that $Y\subset V^h_Z$ is closed subset contained $s(Z)$ and quasifinite over $U$. Then $can\colon Y\rightarrow V$ is closed imbedding , $can(Y)$ contains $Z$ and $can^{-1}(can(Y))=Y$.
\end{mlm}
\begin{proof}
This lemma proved in \cite{GNP}.
\end{proof}

\begin{mlm}\label{lemma_isom_gens}
Let $X$ be a smooth affine variety over field $k$, $Z\subset Y\subset X$ be closed subsets, 
$g\colon X^h_Z\rightarrow X$ be a henselization of $X$ in $Z$. Suppose that there exist the inclusion $j\colon Y\hookrightarrow X^h_Z$ such that the diagram
$$
\xymatrix{X^h_Z\ar[d]_{g}&\\
X&Y\ar@{_(->}[l]^{i}\ar@{_(->}[lu]_{j}\\}
$$
conutes and $g^{-1}(Y)=j(Y)$. Then the canonical map $X_Y^h\rightarrow X_Z^h$ is the isomorphism of schemes.
\end{mlm}

\begin{mrem}
The canonical map $p\colon F_n(U,X)\rightarrow F_n\left(U, \dfrac{X}{X-S}\right)$ induces the map $p\colon F_n(U,X)\rightarrow F_n^{qf}\left(U, \dfrac{X}{X-S}\right)$.
\end{mrem}

\begin{mth}
Let $U$ be a Henselian scheme, $X$ be a smooth variety, $S\subset X$ be a closed subset. Then the next square is cocartesian
$$\xymatrix{F_n(U,X-S)\ar@{^(->}[r]^i\ar[d]&F_n(U,X)\ar[d]^p\\
\ast\ar@{^(->}[r] &F_n^{qf}\left(U, \dfrac{X}{X-S}\right).}$$
\end{mth}

\begin{mcor}
Let $X$ be a smooth variety over field $k$, $S\subset X$ be closed subset. Then there exist the following  exact sequence
$$0\rightarrow \mathbb{Z}F(X-S)\rightarrow\mathbb{Z}F(X)\rightarrow \mathbb{Z}F^{qf}\left(\dfrac{X}{X-S}\right)\rightarrow 0.$$
\end{mcor}

\begin{mcor}
Let $X$ be a smooth variety over field $k$, $S\subset X$ be closed subset. Then the following triangle 
$$\mathbb{Z}F(X-S)\rightarrow\mathbb{Z}F(X)\rightarrow \mathbb{Z}F^{qf}\left(\dfrac{X}{X-S}\right)\rightarrow \mathbb{Z}F(X-S)[1]$$
is distinguished in $D^-(NSZF_*)$.
\end{mcor}

\begin{mlm}\label{simple_presheave_lemma}
Let $\mathcal{F}$ be a homotopy invariat quasistable presheave with $\mathbb{Z}F_*$-transfers, $X$ be a local Henzelian scheme over field $k$. Then
$$H^i_{Nis}(X\times \mathbb{G}_m,\mathcal{F})=
\left\{
\begin{array}{ll}
\mathcal{F}(X\times\mathbb{G}_m)& \mbox{for $i=0$,}\\
\\
0& \mbox{for $i\neq 0$.}\\
\end{array}
\right.
$$
\end{mlm}

\begin{mth}\label{cancelcation}
Let $X=Spec~\mathcal{O}$, where $\mathcal{O}$ is Henselian local ring over field $k$, $Y$ be a smooth variety, $S\subset Y$  be its closed subset. Then the map
$$-\boxtimes (id_{\mathbb{G}_m}-e_1)\colon \mathbb{Z}F^{qf}\left(\Delta^\bullet\times X, \dfrac{Y}{Y-S}\right)\rightarrow \mathbb{Z}F^{qf}\left(\Delta^\bullet\times X\times \mathbb{G}_m^{\wedge 1}, \dfrac{Y}{Y-S}\times\mathbb{G}_m^{\wedge 1}\right)$$
is quasiisomorphism of complexes.
\end{mth}

\begin{proof}
Consider two distinguished triangles in $D^-(NSZF_*)$:
$$1)~\mathbb{Z}F(\Delta^\bullet\times -, Y-S)\rightarrow \mathbb{Z}F(\Delta^\bullet\times -, Y)\rightarrow \mathbb{Z}F^{qf}(\Delta^\bullet\times -, \dfrac{Y}{Y-S}),$$
$$2)~\mathbb{Z}F(\Delta^\bullet\times -, (Y-S)\times \mathbb{G}_m^{\wedge 1})\rightarrow \mathbb{Z}F(\Delta^\bullet\times -, Y\times \mathbb{G}_m^{\wedge 1})\rightarrow \mathbb{Z}F^{qf}(\Delta^\bullet\times -, \dfrac{Y}{Y-S}\times \mathbb{G}_m^{\wedge 1}).$$

From 1) we get the following exact sequence for scheme  $X$
$$0\rightarrow\mathbb{Z}F(\Delta^\bullet\times X, Y-S)\rightarrow \mathbb{Z}F(\Delta^\bullet\times X, Y)\rightarrow \mathbb{Z}F^{qf}(\Delta^\bullet\times X, \dfrac{Y}{Y-S})\rightarrow 0$$

From 2) we get for  $X\times \mathbb{G}_m$:
\begin{align*}
0\rightarrow H^0_{Nis}(X\times \mathbb{G}_m, \mathbb{Z}F(\Delta^\bullet\times -, (Y-S&)\times \mathbb{G}_m^{\wedge 1}))\rightarrow\\ 
\rightarrow H^0_{Nis}(X\times \mathbb{G}_m, \mathbb{Z}F(\Delta^\bullet\times -, Y\times& \mathbb{G}_m^{\wedge 1}))\rightarrow \\
\rightarrow H^0_{Nis}(X\times \mathbb{G}_m, \mathbb{Z}F^{qf}(\Delta^\bullet\times - , &\dfrac{Y}{Y-S}\times \mathbb{G}_m^{\wedge 1}))\rightarrow\\
\rightarrow H^1_{Nis}(X\times \mathbb{G}_m, \mathbb{Z}F(\Delta^\bullet\times -&, (Y-S)\times \mathbb{G}_m^{\wedge 1}))\rightarrow \ldots.
\end{align*}
We'll calculate the hypercohomologies in this sequence. We restrict ourselves to calculating cohomology\break $H^i_{Nis}\left(X\times \mathbb{G}_m, \mathbb{Z}F^{qf}(\Delta^\bullet\times - , \dfrac{Y}{Y-S}\times \mathbb{G}_m^{\wedge 1})\right)$. The remaining cohomologies can be calculated in the same way.

Consider the spectral sequence

\begin{align*}
H^p_{Nis}\left(X\times\mathbb{G}_m, \left(\underline{h}^q(\mathbb{Z}F^{qf}(\Delta^\bullet\times - , \dfrac{Y}{Y-S}\times \mathbb{G}_m^{\wedge 1}))\right)_{Nis}\right)\Longrightarrow\\
\Longrightarrow H^{p+q}_{Nis}\left(X\times \mathbb{G}_m, \mathbb{Z}F^{qf}(\Delta^\bullet\times - , \dfrac{Y}{Y-S}\times \mathbb{G}_m^{\wedge 1})\right),
 \end{align*}
where $\underline{h}^q(\mathbb{Z}F^{qf}(\Delta^\bullet\times - , \dfrac{Y}{Y-S}\times \mathbb{G}_m^{\wedge 1}))$ --- are preasheave cohomologies of the complex $\mathbb{Z}F^{qf}(\Delta^\bullet\times - , \dfrac{Y}{Y-S}\times \mathbb{G}_m^{\wedge 1}))$. On the left there are the usual cohomology of sheves, on the right there are the hypercohomology of the sheaves complex.

Denote the sheaf $\mathbb{Z}F^{qf}( - , \dfrac{Y}{Y-S}\times \mathbb{G}_m^{\wedge 1})$ with $\mathcal{F}$. Note that it is quasistable and is a sheaf with $\mathbb{Z}F_*$-transfers. Then by the theorem from \cite{GP2} preasheves  $h^q(\mathcal{F}(\Delta^\bullet\times-))$ are homotopy invariant, quasistable and has $\mathbb{Z}F_*$-transfers. Then by the theorem from  \cite{GP2} sheaves $\underline{h}^q(\mathcal{F}(\Delta^\bullet\times-))$ are also homotopy invariant , quasistable and has $\mathbb{Z}F_*$-transfers. Then by lemma \ref{simple_presheave_lemma} we get
$$H^p_{Nis}\left(X\times\mathbb{G}_m, \left(\underline{h}^q(\mathbb{Z}F^{qf}(\Delta^\bullet\times - , \dfrac{Y}{Y-S}\times \mathbb{G}_m^{\wedge 1}))\right)_{Nis}\right)=$$
$$
=\left\{
\begin{array}{ll}
\underline{h}^q(\mathbb{Z}F^{qf}(\Delta^\bullet\times X\times \mathbb{G}_m , \dfrac{Y}{Y-S}\times \mathbb{G}_m^{\wedge 1}))& \mbox{for $i=0$,}\\
\\
0& \mbox{for $i\neq 0$.}\\
\end{array}
\right.
$$
That is, the spectral sequence degenerates at the first step, and exactly one nonzero group falls into the corresponding dimension of the limit. Then
$$H^{n}_{Nis}\left(X\times \mathbb{G}_m, \mathbb{Z}F^{qf}(\Delta^\bullet\times - , \dfrac{Y}{Y-S}\times \mathbb{G}_m^{\wedge 1})\right)=\underline{h}^q\left(\mathbb{Z}F^{qf}(\Delta^\bullet\times X\times \mathbb{G}_m , \dfrac{Y}{Y-S}\times \mathbb{G}_m^{\wedge 1})\right).$$
Considering the direct summands on the left and on the right we get  
$$H^{n}_{Nis}\left(X\times \mathbb{G}_m^{\wedge 1}, \mathbb{Z}F^{qf}(\Delta^\bullet\times - , \dfrac{Y}{Y-S}\times \mathbb{G}_m^{\wedge 1})\right)=\underline{h}^q\left(\mathbb{Z}F^{qf}(\Delta^\bullet\times X\times \mathbb{G}_m^{\wedge 1} , \dfrac{Y}{Y-S}\times \mathbb{G}_m^{\wedge 1})\right).$$
Similar calculation is true for sheaves $\mathbb{Z}F(\Delta^\bullet\times -, (Y-S)\times \mathbb{G}_m^{\wedge 1})$,\break  $\mathbb{Z}F(\Delta^\bullet\times -, Y\times \mathbb{G}_m^{\wedge 1})$, so threre is a long exact sequence
\begin{align*}
0\rightarrow \underline{h}^0(\mathbb{Z}F(\Delta^\bullet \times X\times \mathbb{G}_m^{\wedge 1}, (Y-&S)\times \mathbb{G}_m^{\wedge 1}))\rightarrow\\ 
\rightarrow \underline{h}^0((\mathbb{Z}F(\Delta^\bullet\times X\times& \mathbb{G}_m^{\wedge 1}, Y\times \mathbb{G}_m^{\wedge 1}))\rightarrow \\
\rightarrow \underline{h}^0(\mathbb{Z}F^{qf}(\Delta^\bullet\times &X\times \mathbb{G}_m^{\wedge 1} , \dfrac{Y}{Y-S}\times \mathbb{G}_m^{\wedge 1}))\rightarrow\\
\rightarrow \underline{h}^1(\mathbb{Z}F(\Delta^\bullet &\times X\times \mathbb{G}_m^{\wedge 1} , (Y-S)\times \mathbb{G}_m^{\wedge 1}))\rightarrow \ldots.
\end{align*}
We'll show tat the map $$-\boxtimes (id_{\mathbb{G}_m}-e_1)\colon \mathbb{Z}F^{qf}(\Delta^\bullet\times X, \dfrac{Y}{Y-S})\rightarrow \mathbb{Z}F^{qf}(\Delta^\bullet\times X\times \mathbb{G}_m^{\wedge 1}, \dfrac{Y}{Y-S}\times\mathbb{G}_m^{\wedge 1})$$ induces the isomorphisms of complexes.
The map $-\boxtimes (id_{\mathbb{G}_m}-e_1)$ induces the triangle morphism, and so induces a map of long exact sequences:
$$\xymatrix{
\ldots\ar[d]&\ldots\ar[d]\\
\underline{h}^n(\mathbb{Z}F(\Delta^\bullet \times X, (Y-S)))\ar[d]\ar[r]&\underline{h}^n(\mathbb{Z}F(\Delta^\bullet \times X\times \mathbb{G}_m^{\wedge 1}, (Y-S)\times \mathbb{G}_m^{\wedge 1}))\ar[d]\\ 
\underline{h}^{n}(\mathbb{Z}F(\Delta^\bullet\times X, Y))\ar[d]\ar[r]&\underline{h}^n(\mathbb{Z}F(\Delta^\bullet\times X\times \mathbb{G}_m^{\wedge 1}, Y\times \mathbb{G}_m^{\wedge 1}))\ar[d]\\
\underline{h}^n\left(\mathbb{Z}F^{qf}\left(\Delta^\bullet\times X, \dfrac{Y}{Y-S}\right)\right)\ar[d]\ar[r]&\underline{h}^n\left(\mathbb{Z}F^{qf}\left(\Delta^\bullet\times X\times \mathbb{G}_m^{\wedge 1} , \dfrac{Y}{Y-S}\times \mathbb{G}_m^{\wedge 1}\right)\right)\ar[d]\\
 \underline{h}^{n+1}(\mathbb{Z}F(\Delta^\bullet \times X , (Y-S)))\ar[d]\ar[r]&
  \underline{h}^{n+1}(\mathbb{Z}F(\Delta^\bullet \times X\times \mathbb{G}_m^{\wedge 1} , (Y-S)\times \mathbb{G}_m^{\wedge 1}))\ar[d]\\
  \underline{h}^{n+1}(\mathbb{Z}F(\Delta^\bullet \times X, (Y-S)))\ar[d]\ar[r]&\underline{h}^{n+1}(\mathbb{Z}F(\Delta^\bullet \times X\times \mathbb{G}_m^{\wedge 1}, (Y-S)\times \mathbb{G}_m^{\wedge 1}))\ar[d]\\ 
  \ldots&\ldots.
}
$$
Horizontal maps except maybe
$$\underline{h}^n\left(\mathbb{Z}F^{qf}\left(\Delta^\bullet\times X, \dfrac{Y}{Y-S}\right)\right)\rightarrow\underline{h}^n\left(\mathbb{Z}F^{qf}\left(\Delta^\bullet\times X\times \mathbb{G}_m^{\wedge 1} , \dfrac{Y}{Y-S}\times \mathbb{G}_m^{\wedge 1}\right)\right)$$
are isomorphisms by theorem C from \cite{AGP}. Therefore, the map of our interest is also an isomorphism.
\end{proof}

\end{document}